\documentclass[letterpaper, 10 pt, conference]{ieeeconf}  

\IEEEoverridecommandlockouts                              

\pdfoutput=1

\usepackage{amsmath,amsfonts,amssymb}
\usepackage{graphics} 
\usepackage{epsfig,color} 
\usepackage{times} 
\usepackage{amssymb}  

\usepackage{amsthm}
\usepackage{graphicx}
\graphicspath{ {./images/} }
\usepackage{graphicx}
\usepackage{booktabs} 
\usepackage{subfig}
\usepackage[ruled,vlined]{algorithm2e}

\newtheorem{notation}{\bf{Notation}}

\newcommand{\Rset}{\mathbb{R}}

\newcommand{\Vcal}{{\cal V}}

\newcommand{\Ebf}{{\bf E}}

\newcommand{\Lbf}{{\bf L}}
\newcommand{\Mbf}{{\bf M}}

\newcommand{\Pbf}{{\bf P}}

\newcommand{\Ubf}{{\bf U}}

\newcommand{\Wbf}{{\bf W}}
\newcommand{\Xbf}{{\bf X}}
\newcommand{\Ybf}{{\bf Y}}

\newcommand{\ebf}{{\bf e}}

\newcommand{\xbf}{{\bf x}}
\newcommand{\ybf}{{\bf y}}

\newcommand{\1}{{\mathbf{1}}}

\newcommand{\xbar}{{\bar{x}}}

\newtheorem{lem}{Lemma}
\newtheorem{thm}{Theorem}

\newtheorem{assump}{Assumption}
\newtheorem{remark}{Remark}

\usepackage[bookmarks=true,pageanchor,colorlinks,linkcolor=red,anchorcolor=blue, citecolor=blue,urlcolor=blue,hyperfootnotes=false]{hyperref}
\title{\LARGE \bf
Convergence Rates of Decentralized Gradient Methods \\over Cluster Networks
}

\author{Amit Dutta \quad Nila Masrourisaadat\quad Thinh T. Doan 
\thanks{The authors are with the Bradley Department of Electrical and Computer Engineering, Virginia Tech, Blacksburg, VA. Email: {\tt\small \{amitdutta, nilamasrouri, thinhdoan\}@vt.edu}}%
}

\begin{document}

\maketitle
\thispagestyle{empty}
\pagestyle{empty}

\begin{abstract}

We present an analysis for the performance of decentralized consensus-based gradient (DCG) methods for solving optimization problems over a cluster network of nodes. This type of network is composed of a number of densely connected clusters with a sparse connection between them. Decentralized algorithms over cluster networks have been observed to constitute two-time-scale dynamics, where information within any cluster is mixed much faster than the one across clusters. Based on this observation, we present a novel analysis to study the convergence of the DCG methods over cluster networks. In particular, we show that these methods converge at a rate $\ln(T)/T$ and only scale with the number of clusters, which is relatively small to the size of the network. Our result improves the existing analysis, where these methods are shown to scale with the size of the network. The key technique in our analysis is to consider a novel Lyapunov function that captures the impact of multiple time-scale dynamics on the convergence of this method. We also illustrate our theoretical results by a number of numerical simulations using DCG methods over different cluster networks. 
\end{abstract}

\section{Introduction}
In this paper, we study a distributed optimization problem over a network of nodes, where the objective is composed of local functions known by the nodes. For solving this problem, we are interested in using the classic decentralized consensus-based gradient (DCG) methods  \cite{Nedic_review2018}, where the nodes are only allowed to communicate with their neighboring nodes. Our focus is to study the convergence rate of this method when the underlying communication network shared between nodes has a cluster structure. In particular, the network is composed of many densely connected clusters, where there are sparse connection between these clusters. Fig. \ref{fig:clusters} illustrates an example of this network.   

Cluster networks are ubiquitous in large-scale systems, for example, in small-world networks \cite{watts1998collective}, wireless sensor networks \cite{zhao2004wireless} and power systems networks \cite{chow2013power}. In cluster networks, information exchanged between nodes often constitute two-time-scale dynamics \cite{chow1985time}. More specifically, the dynamics of the local interactions of the nodes within clusters evolve at a faster time scale (due to dense communication) than the slow aggregate dynamics (due to sparse communication) across clusters. This observation has been utilized in a number of works to study the behavior of the popular distributed consensus methods over cluster networks in different applications; see for example \cite{chow1985time, martin2016time, biyik2008area,boker, mukherjee2021reduced, boker2016aggregate,Mor2016,Pham2020c,ThiemDN2020,Dutta_cluster_consensus2021} and the references therein. We are, however, not aware of any prior works in studying the convergence properties of DCG methods over cluster networks. The focus of this paper is to fill this gap.   

\textbf{Main contributions.} The main contribution of this paper is to characterize the performance of a continuous-time version of DCG method over cluster networks. In particular, we provide an explicit formula on the rate of convergence of this method with respect to the topology of the underling network. Our result shows that under a reasonable assumption often considered in the literature the convergence time of this method only scales with the number of clusters (which is much smaller than the number of nodes in the network). We support our theoretical results by a number of numerical simulations in using DCG methods for solving distributed least-square optimization problems over different cluster networks.       

\begin{figure}[t]
    \centering
    \vspace{0.4cm}
    \includegraphics[width=0.8\columnwidth]{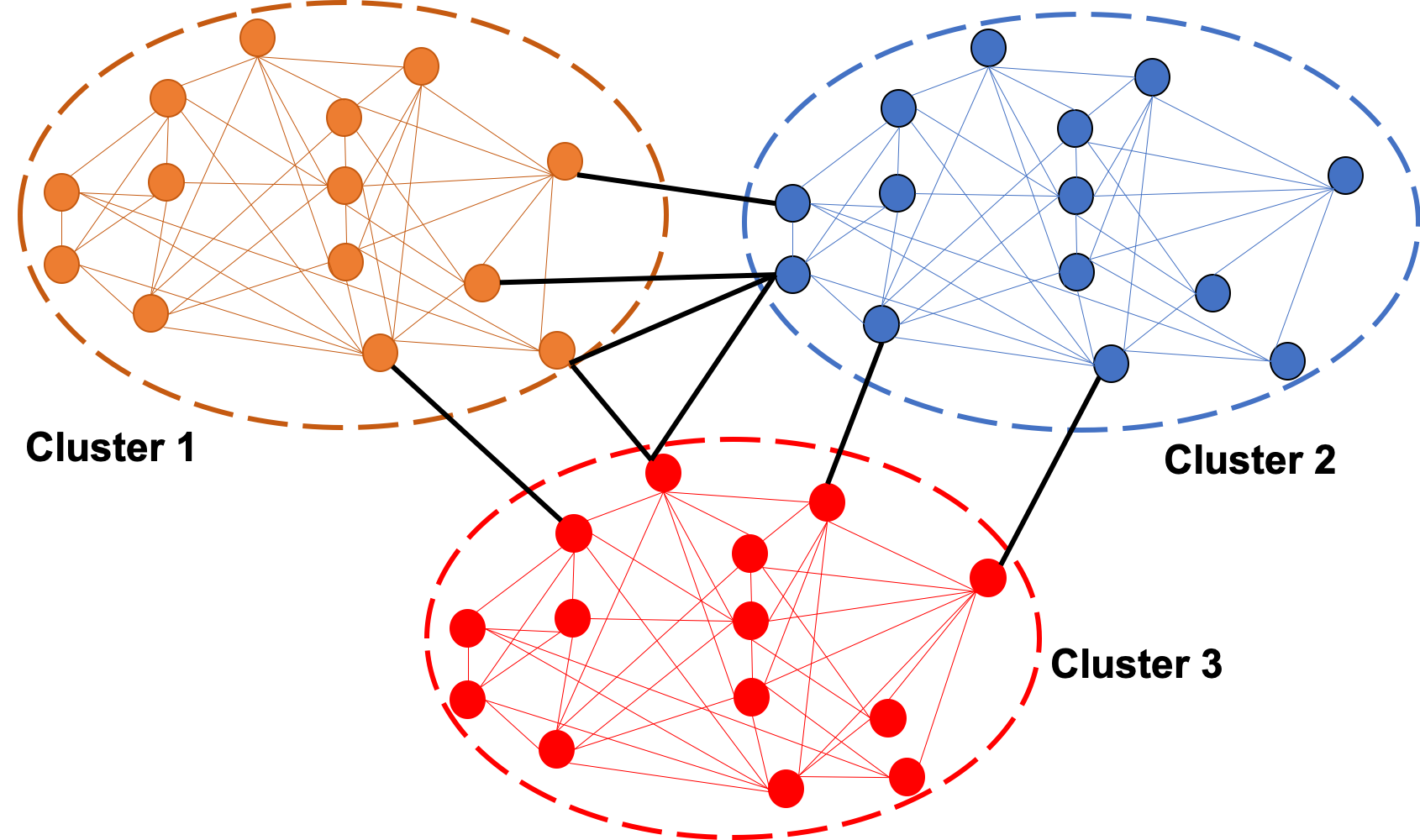}
    \caption{$42$-nodes network is partitioned into three densely connected clusters with sparse connections between them.}
    \label{fig:clusters}
    \vspace{-0.3cm}
\end{figure}
\textbf{Technical Approach.} The key idea in our analysis is to introduce a new composite Lyapunov function with respect to the time-scale separation in the network, inspired by singular perturbation theory \cite{kalil2} and recent analysis for the centralized two-time-scale stochastic approximation \cite{Doan2021_SIAM_TwoTimeScaleSA,Doan_two_time_SA2020,zeng2021two}. Our approach is different than the typical singular perturbation approach reported in \cite{chow1985time,boker2016aggregate,biyik2008area} in that it is not based on reducing the system model into two smaller models. This means that our approach does not require the system to be modeled in the standard singular perturbation form, which requires the fast subsystem to have distinct (isolated) real roots. 
   
We note that the convergence of DCG methods has been studied extensively in the literature \cite{Nedic_review2018}. However, the existing analysis only shows that the convergence time of this method scales with the number of nodes, which is much larger than the number of clusters. It means that our analysis provides a much tighter result than the existing works. In particular, our analysis provides more insights how the nodes interact within their cluster versus their external communication with other nodes in other clusters. Knowing how the dynamics within and across clusters evolve provides an useful approach to design better distributed control strategies in different applications including robotics and power networks.


\begin{notation}
We use boldface to distinguish between vectors $\mathbf{x} \in \mathbb{R}^{n}$ and scalars $x$ in $\mathbb{R}$ for some $n$. Given $\mathbf{x} \in \mathbb{R}^{n}$, we write $\mathbf{x}  = \begin{bmatrix}
x_{1} & x_{2} & \cdots & x_{n}
\end{bmatrix}^{T}$ and let $\|\mathbf{x}\|$ denote its Euclidean norm. 
We use boldface for a matrix $\mathbf{M}$ and denote by $\|\Mbf\|$ its $2$-norm. In addition, we denote by $\sigma_{2}(\mathbf{M})$ the second smallest eigenvalue of $\mathbf{M}$. Finally, given $N>0$ we denote by $\mathbf{1}_{N}$ a vector in $\mathbb{R}^{N}$ whose entries are equal to $1$. 
\end{notation}

\section{Problem Formulation}
We consider an optimization problem where the objective function is distributed over a network of $N$ nodes. In particular, associated with each node $i$ is a function $f_{i}:\mathbb{R}^{d}\rightarrow\mathbb{R}$. The goal of the nodes is to collaboratively solve
\begin{align}\label{eq:objective}
\underset{\mathbf{x}\in\mathbb{R}^{d}}{\text{minimize }}  f(\mathbf{x}) \triangleq  \frac{1}{N}\sum_{i=1}^{n}f_{i}(\mathbf{x}).
\end{align}
In this setting, we assume that the nodes can communicate with each other over an undirected and connected graph $G = (V,E)$. Here, $V = \{1,\ldots,N\}$ and $E = V\times V$ are the vertex and edge sets, respectively. Moreover, nodes $i$ and $j$ can exchange messages with each other if and only if $(i,j) \in E$. Let $N_{i} = \{j\,|\, (i,j) \in E\}$ be the neighboring set of node $i$. 

For solving problem \eqref{eq:objective}, we consider a continuous-time variant of popular decentralized consensus-based gradient (DCG) method \cite{Nedic_review2018}. In this method, each node $i$ maintains a local variable $\mathbf{x}_{i}$ to estimate $\mathbf{x}^{\star}$, an optimal solution of \eqref{eq:objective}. Each node $i$ then iteratively updates its iterate as
\begin{align}
\dot{\mathbf{x}}_{i}(t) = -\sum_{j\in N_{i}}L_{ij} \mathbf{x}_{j}(t) - \gamma(t)\nabla f_{i}(\mathbf{x}_{i}(t)),     \label{alg:DCG}
\end{align}
where $L_{ij}$ is the $(i,j)$-entry of the Laplacian matrix associated with $G$, and $\gamma$ is some positive step sizes.   

In this paper, we are interested in studying the performance of DCG method when the communication graph $G$ has a cluster structure. In particular, we consider the case when $G$ is divided into $r$ densely connected clusters $C_{1},C_{2},\ldots,C_{r}$ while there are sparse connections between these clusters. Fig. \ref{fig:clusters} illustrates an example of cluster networks, where there are 3 densely connected clusters with a small number of edges across clusters. For each cluster $C_{\alpha}$, $\alpha = 1,\ldots, r$, let $G_{\alpha} = (V_{\alpha},E_{\alpha})$ be the graph representing the connection between nodes in $C_{\alpha}$. In addition, let $G_{E}$ be the graph describing the external connections between clusters.

Distributed consensus methods, a special of DCG,  over cluster networks have been observed to constitute two-time-scale dynamics, where the information within any cluster is mixed much faster (due to dense communications) than information diffused from one cluster to another (due to sparse communications) \cite{biyik2008area,boker, boker2016aggregate,Pham2020c,ThiemDN2020,Dutta_cluster_consensus2021}. In this paper, we show that this observation continues to hold in the context of DCG method. In addition, we will show that the convergence rates of this method over cluster networks only scale with the number of clusters, which is much smaller than the number of nodes.

We conclude this section with the following assumptions, which facilitate our analysis presented in the next section.

\begin{assump}\label{assump:graph_connectivity}
Each internal graph $G_{\alpha}$, for all $\alpha = 1,\ldots, r$, and $G_{E}$ are connected and undirected. 
\end{assump}

\begin{assump}\label{assump:obj}
The functions $f_{i}$'s are Lipschitz continuous with a constant $L > 0$ 
\begin{align}
\| f_{i}(\mathbf{x}) -  f_{i}(\mathbf{y})\| &\leq L\|\mathbf{x}-\mathbf{y}\|,\quad \forall \mathbf{x},\mathbf{y}.\label{assump:obj:func_Lipschitz}
\end{align}
Moreover, $f_{i}$'s are strongly convex with a constant $\mu > 0$
\begin{align}
 \frac{\mu}{2}\|\mathbf{x}-\mathbf{y}\|^2 \leq f(\mathbf{x}) - f(\mathbf{y}) - \nabla f(\mathbf{y})^T(\mathbf{x}-\mathbf{y}), \forall \mathbf{x},\mathbf{y}. \label{assump:obj:sc} 
\end{align}
\end{assump}

\begin{remark}
For an ease of exposition, we will consider a scalar setting, i.e., $x,x_{i}\in\mathbb{R}$. An extension to the multi-dimensional case is trivial, which we will discuss later.   
\end{remark}


\section{Dynamics of DCG over Cluster Networks}
In this section, we reformulate the update of DCG in \eqref{alg:DCG} corresponding to the cluster structure of $G$. We first introduce some notation as follows.
\begin{align*}
\mathbf{x} &= [x_{1},  x_{2}, \cdots, x_{N}] ^T,\\ 
\nabla F(\mathbf{x}) &= [
    \nabla f_{1}(x_{1}), \nabla f_{2}(x_{2}), \cdots, \nabla f_{N}(x_{N})]^T.   
\end{align*}
To the rest of this paper, we will often drop the time notation $t$ in the variables for brevity.\\ 
\noindent\textbf{Cluster Structure:} For each cluster $C_{\alpha}$, we denote by  $\mathbf{L}_{\alpha}\in\mathbb{R}^{|V_{\alpha}|\times |V_{\alpha}|}$ the Laplacian matrix corresponding to graph $G_{\alpha}$ of $C_{\alpha}$. In addition, let $\mathbf{L}^{I}$ be
\begin{align*}
    \mathbf{L}^{I} = \text{diag}\{\mathbf{L}_{1}, \mathbf{L}_{2},...,\mathbf{L}_{r}\} \in \mathbb{R}^{N \times N},
\end{align*}
Similarly, we denote by $\mathbf{L}^{E}\in \mathbb{R}^{N\times N}$ the sparse Laplacian matrix corresponding to the external graph $G_{E}$. Then we have
\[
\mathbf{L} = \mathbf{L}^{I} +  \mathbf{L}^{E},
\]
which by \eqref{alg:DCG} gives   \begin{align}\label{Eq:Consesus_dynamics_Laplacian}
     \Dot{\mathbf{x}}  &= -(\mathbf{L}^{I} + \mathbf{L}^{E})\mathbf{x} -\gamma(t) \nabla F(\mathbf{x}).
\end{align}
This update has two parts, namely the consensus step (the first term) and local gradient step (the second term). The goal of the first step is to push the local estimates at the nodes to a common point while the second step is to push this point to the solution $x^*$ of \eqref{eq:objective}. In cluster networks, the consensus step has been observed to constitute two-time-scale dynamics, which we present as two terms $-\mathbf{L}^{I}\mathbf{x}$ and $-\mathbf{L}^{E}\mathbf{x}$. Basically, the Laplacian $\mathbf{L}^{I}$ of the internal graphs is much denser than the one of the external graph $\mathbf{L}^{E}$, thus, one can expect that the first term decays to zero much faster than the second term. In the next section, we present the so-called fast and slow variables to characterize the dynamics of this two-time-scale phenomenon.       

\subsection{Fast and slow variables in clusters networks}
We now introduce two variables to present the fast and slow dynamics in cluster networks. These variables defined below are based on the difference in the way information is mixed within any cluster versus the one diffused from one cluster to another cluster. We note that the fast and slow variables introduced in this paper are fundamentally different from the ones used in \cite{biyik2008area} and \cite{chow1985time}.

\subsubsection{Slow variable}
To present our motivation, let us assume that $\gamma(t)=0$. In this case, \eqref{Eq:Consesus_dynamics_Laplacian} is the classic consensus update. Due to the dense communication in each cluster, its variables will quickly converge to the same value, which is their average. On the other hand, the sparse communication between these clusters is to slowly push different common values in the clusters to a common point. Based on this observation, we define the slow variable in each cluster $\gamma(t)$ as the average of its nodes' variables, i.e., let $\bar{x}_{\alpha}$ be the slow variable of $C_{\alpha}$ defined as
\begin{align*}
     \bar{x}_{\alpha} = \frac{1}{N_{\alpha}}\mathbf{1}_{N_{\alpha}}^{T}\mathbf{x}_{\alpha},
\end{align*}
where $\mathbf{1}_{N_{\alpha}} = [1,1,...,1]^{T} \in\mathbb{R}^{N_{\alpha}}$ and $N_{\alpha}$ is the number of nodes in $C_{\alpha}$. In addition, $\mathbf{x}_{\alpha}$ is the vector in $\mathbb{R}^{N_{\alpha}}$, whose entries are $x_{i}$ for $i\in C_{\alpha}$. Next, denote by $\mathbf{U}$ and $\mathbf{P}$ as
\begin{align}
\begin{aligned}
\mathbf{U}  &=  diag(\mathbf{1}_{N_1},\mathbf{1}_{N_2},...\mathbf{1}_{N_r}) \in \mathbb{R}^{N \times r},\\ 
\mathbf{P} &= diag(N_{1},N_{2},...N_{r}) \in \mathbb{R}^{r \times r}.
\end{aligned}\label{notation:U_P}
\end{align}
Note that $\mathbf{P} = \mathbf{U}^{T}\mathbf{U}$. Moreover, let $\ybf$ be defined as 
\begin{align*}
\ybf \triangleq [\bar{x}_{1},\bar{x}_{2},\ldots,\bar{x}_{r}]^T \in \mathbb{R}^{r},
\end{align*}
which is the slow vector of the entire network and satisfy
\begin{align}\label{Eq:slow_variable_nw}
    \ybf = \mathbf{P}^{-1}\mathbf{U}^{T}\mathbf{x}.
\end{align}
It follows from the formulation that the slow variable represents the states of the nodes of the aggregate network derived from the cluster network. We denote $\widetilde{G}_{E}$ as the external aggregate graph. Note that $\mathbf{L}^{E}$ is a sparse matrix due to the cluster structure of our graph. We denote by $\widetilde{\mathbf{L}}^{E} \in \mathbb{R}^{r \times r}$ the Laplacian corresponding $\widetilde{G}_{E}$. We consider the following useful result about the relation between $\Lbf^{E}$ and  $\widetilde{\mathbf{L}}^{E}$.
\begin{lem}\label{lem:L_tilde}
The Laplacian $\widetilde{\mathbf{L}}^{E}$ satisfies
\begin{align}
    \widetilde{\mathbf{L}}^{E} = \mathbf{U}^{T}\mathbf{L}^{E}\mathbf{U} .\label{lem:L_tilde:eq}
\end{align}
\end{lem}

\begin{proof}
Observe that 
\begin{align*}
   \mathbf{L}^{E} =  \begin{bmatrix}
    \mathbf{L}^{E}_{11} & \mathbf{L}^{E}_{12} & \cdots & \mathbf{L}^{E}_{1r}\\
    \mathbf{L}^{E}_{21} & \mathbf{L}^{E}_{22} & \cdots & \mathbf{L}^{E}_{2r}\\
    \vdots & \vdots & \ddots & \vdots\\
    \mathbf{L}^{E}_{r1} & \mathbf{L}^{E}_{r2} & \cdots & \mathbf{L}^{E}_{rr}
    \end{bmatrix},
\end{align*}
where $\mathbf{L}^{E}_{\alpha\alpha}$ is a diagonal matrix associated with the cluster $\alpha$ whose the diagonal elements are the external degree of each node (the number of external connections). Note that these diagonal elements can be equal to zero since most of the nodes may not have any external connection. In addition, the off-diagonal matrices represent the inter-cluster edges between clusters. Next using $\mathbf{U}$ in \eqref{notation:U_P} we have
\[\mathbf{U}^{T}\mathbf{L}^{E}\mathbf{U} = \Big[\mathbf{\mathbf{1}_{N_{\alpha}}^T\mathbf{L}_{\alpha\beta}^{E}\mathbf\mathbf{1}_{N_{\beta}}\Big]\in \mathbb{R}^{r\times r} }, \text{ for } \alpha,\beta = 1,\ldots,r.\] Here, $1^{T}_{N_{\alpha}}\mathbf{L}^{E}_{\alpha\alpha}\mathbf{1}_{N_{\alpha}}$ is the sum of all the external degrees  associated with the cluster $\alpha$. Also, the off-diagonal elements are the sum of the inter-cluster edges. Thus, $\mathbf{U}^{T}\mathbf{L}^{E}\mathbf{U}$ represents the Laplacian of $G_{E}$.



\end{proof}

\subsubsection{Fast variable}
Based on the argument above, we define the fast variables within any cluster as the relative difference of its estimates to its slow variable. In particular, we define the fast variable $e^{x_{i}}_{\alpha}$ in $C_{\alpha}$ as
\begin{align}\label{Eq:fast_variable_node}
    e^{x_i}_{\alpha} = x_{i} - \bar{x}_{\alpha},\quad \text{for all } i\in V_{\alpha}. 
\end{align}
For each cluster $\alpha$, we denote by $\mathbf{W}_{\alpha}$ the diagonally dominant centering matrix
\begin{align}
\mathbf{W}_{\alpha} = \left(\mathbf{I}_{N_{\alpha}} - \frac{1}{N_{\alpha}}\mathbf{1}_{N_{\alpha}}\mathbf{1}_{N_{\alpha}}^{T}\right) \in \mathbb{R}^{N_{\alpha}\times N_{\alpha}},\label{notation:W_alpha} 
\end{align}
and let $\mathbf{e}^{x}_{\alpha}$ be a vector in $\mathbb{R}^{N_{\alpha}}$, whose i-th entry is $e^{x_{i}}_{\alpha}$. In view of \eqref{Eq:fast_variable_node} \eqref{notation:W_alpha} and we have
\begin{align}
     \mathbf{e}^{x}_{\alpha} &= \mathbf{x}_{\alpha} - \bar{x}_{\alpha}\mathbf{1}_{N_{\alpha}} = \mathbf{W}_{\alpha}\mathbf{x}_{\alpha}. 
     \label{Eq:fast_variable_area_centering}
\end{align}
Finally, we denote by $\mathbf{W} = diag (\mathbf{W}_{\alpha}) \in \mathbb{R}^{N\times N}$ and $\ebf^{x}$ as
\begin{align*}
\mathbf{e}^{x} = [(\ebf^{x}_{1})^{T},...,(\ebf^{x}_{r})^{T}]^T \in \mathbb{R}^{N}.
\end{align*}
Thus, the fast variable for the entire network $G$ is given by
\begin{align}
\mathbf{e}^{x} &= \mathbf{x} - \mathbf{U}\mathbf{y} = \mathbf{W}\mathbf{x}. \label{Eq:fast_variable_area_v2}
\end{align}
\begin{remark}
It is worth noting that the definition of the fast variables \eqref{Eq:fast_variable_node} are different than that in \cite{chow1985time,biyik2008area}, where they are defined as the relative difference in state values with respect to a reference node (typically the first node). Our definition results in a simpler representation of the system dynamics as there is no need to use a complex similarity transformation as in \cite{chow1985time,biyik2008area}. In addition, formulation \eqref{Eq:fast_variable_node} will help to characterize explicitly the rates of the algorithm, which may not be obvious to derive from the ones in  \cite{chow1985time,biyik2008area}.
\end{remark}

\subsubsection{Fast and Slow Dynamics}
We next present the dynamics for the fast and slow variables based on \eqref{Eq:slow_variable_nw} and \eqref{Eq:fast_variable_area_v2}.

\begin{lem}\label{lem:fast_slow_dynamics}
The fast variable $\mathbf{e}^{x}$  satisfies 
\begin{align}
\hspace{-0.3cm}\dot{\mathbf{e}}^{x} =  -\mathbf{W}(\mathbf{L}^{I}+\mathbf{L}^{E})\mathbf{e}^{x} - \mathbf{W}\mathbf{L}^{E}\mathbf{U}\mathbf{y} - \gamma(t)\mathbf{W}\nabla F(\mathbf{x}),  \label{lem:fast_slow_dynamics:Eq_z}
\end{align}
\end{lem}

\begin{proof}
Using \eqref{Eq:fast_variable_area_v2} and \eqref{Eq:Consesus_dynamics_Laplacian} gives
\begin{align}
    \dot{\mathbf{e}}^{x} &= \mathbf{W}\dot{x} =  - \mathbf{W}(\mathbf{L}^{I} + \mathbf{L}^{E})\xbf - \gamma(t)\Wbf\nabla F(\xbf),\notag\\
    &= -\mathbf{W}(\mathbf{L}^{I} + \mathbf{L}^{E})\ebf^{x} - \mathbf{W}(\mathbf{L}^{I} + \mathbf{L}^{E})\mathbf{U}\ybf\notag\\
    &\quad - \gamma(t)\mathbf{W}\nabla F(\mathbf{x}),
\end{align}
immediately gives \eqref{lem:fast_slow_dynamics:Eq_z} since $\mathbf{L}^{I}\mathbf{U} = 0$.
\end{proof}
\begin{lem}
The  slow variable $\mathbf{y}$ satisfies
\begin{align}\label{lem:fast_slow_dynamics:Eq_y}
    \hspace{-0.2cm}\dot{\mathbf{y}} &= - \mathbf{P}^{-1}\widetilde{\Lbf}^{E}\mathbf{y} - \mathbf{P}^{-1}\mathbf{U}^{T}\mathbf{L}^{E}\mathbf{e}^{x}  -\gamma(t)\mathbf{P}^{-1}\mathbf{U}^{T}\nabla F(\mathbf{x}).  
\end{align}

\end{lem}

\begin{proof}
In view of \eqref{Eq:slow_variable_nw} and \eqref{Eq:Consesus_dynamics_Laplacian}, we obtain
\begin{align}
     \dot{\mathbf{y}}  &= \mathbf{P}^{-1}\mathbf{U}^{T}\dot{\mathbf{x}} \notag\\
    &= - \mathbf{P}^{-1}\mathbf{U}^{T}( \mathbf{L}^{I} + \mathbf{L}^{E})\mathbf{x} -  \gamma(t)\mathbf{P}^{-1}\mathbf{U}^{T}\nabla F(\mathbf{x}), 
\end{align}
Next, using \eqref{Eq:fast_variable_area_v2} we have
\begin{align}
     \dot{\ybf}  &= - \mathbf{P}^{-1}\mathbf{U}^{T}( \mathbf{L}^{I} + \mathbf{L}^{E})(\ebf^{x} + \mathbf{U}\ybf)\notag\\
     &\quad -  \gamma(t)\mathbf{P}^{-1}\mathbf{U}^{T}\nabla F(\mathbf{x}), 
\end{align}
which since $\mathbf{U}^T\mathbf{L}^{I} = \mathbf{L}^{I}\mathbf{U} = 0$ and by \eqref{lem:L_tilde:eq} yields \eqref{lem:fast_slow_dynamics:Eq_y}. 

\end{proof}

\subsection{Inter-cluster variable}
We denote by $e^{y}_{\alpha}$ the difference between the slow variable $\bar{x}_{\alpha}$ of cluster $C_{\alpha}$ to the average of the entire network 
\begin{align*}
    e^{y}_{\alpha} = \bar{x}_{\alpha} - \bar{x},
\end{align*}
where $\bar{x} = \frac{1}{N}\sum_{i = 1}^{N}x_{i}$. We denote by $\ebf^{y}$ the inter-cluster variable for the entire network 
\[\mathbf{e}^y= [e^{y}_{1},...,e^{y}_{r}]^T \in \mathbb{R}^{r},\] 
which can also be expressed as 
\begin{align}\label{Eq:e_y_final}
    \ebf^{y} &= \mathbf{y} - \mathbf{1}_{r}\bar{x}.
\end{align}
Recall that $r$ is the number of areas. The dynamic of this inter-cluster variable is given by the following lemma.



\begin{lem}\label{dynamics_inter_cluster_lem}
The inter cluster variables satisfy the following
\begin{align}\label{eq:dynamics_inter_cluster}
    \dot{\mathbf{e}}_{y} &= -\mathbf{P}^{-1}\widetilde{\Lbf}^{E}\ebf^{y} - \mathbf{P}^{-1}\mathbf{U}^{T}\mathbf{L}^{E}\mathbf{e}^{x} \notag\\ 
    &\quad -\gamma(t)\mathbf{P}^{-1}\mathbf{U}^{T}\nabla F(\mathbf{x})  -\frac{\gamma(t)}{N}\mathbf{1}_{N}^T\nabla F(\mathbf{x})\mathbf{1}_{r},
\end{align}
\end{lem}

\begin{proof}
Using \eqref{Eq:e_y_final} we have
\begin{align*}
    \dot{\mathbf{e}}_{y} = \dot{\mathbf{y}} - \mathbf{1}_{r}\dot{\bar{x}}.
\end{align*}
Using \eqref{Eq:Consesus_dynamics_Laplacian} and since $\mathbf{1}_{N}^T\mathbf{L} = 0$, the preceding relation gives
\begin{align*}
    \dot{\bar{x}} = \frac{1}{N}\mathbf{1}^{T}_{N}\dot{\mathbf{x}} &= - \frac{\gamma(t)}{N}\mathbf{1}^{T}_{N}\nabla F(\mathbf{x}),
\end{align*}
which together with \eqref{lem:fast_slow_dynamics:Eq_y} immediately gives \eqref{eq:dynamics_inter_cluster}.
\end{proof}

\section{Main results}
In this section, we present the main result of this paper, where we study the convergence rate of DCG methods over clustered networks. In particular, we will show that the convergence time of this method only scales with the number of clusters $r$, which is much smaller than the number of nodes $N$. To do that we will consider the following three Lyapunov functions corresponding to the fast, inter-cluster, and optimal residual variables, respectively, 
\begin{align}\label{notation:Lyapunov_func}
\begin{aligned}
    \hspace{-0.25cm}V(\mathbf{e}^{y}) &= \|\mathbf{e}^{y}\|,\;\; V(\mathbf{e}^{x}) = \|\mathbf{e}^{x}\|,\;\;
    V(\bar{x}) = \|\bar{x}-x^{\star}\|^2.
    \end{aligned}
\end{align}
In addition, we denote by $N_{\min}$ and $N_{\max}$
\begin{align}
N_{\min} = \min_{\alpha} N_{\alpha},\quad N_{\max} = \max_{\alpha}N_{\alpha}. \label{notation:Nmin_max}   
\end{align}
We denote by 
\begin{align}
\sigma_{2}^{I} = \min_{\alpha} \sigma_{2}(\Lbf_{\alpha}),\label{notaion:sigma_2_I}
\end{align} 
which is strictly positive since $\Lbf_{\alpha}$ is connected. In addition, let $\sigma_{2}(\widetilde{\Lbf}^{E})$ be the second smallest eigenvalue of $\widetilde{\Lbf}^{E}$, which is also strictly positive since $G_{E}$ is connected. We will use the following observation in this section. The Laplacian $\Lbf_{\alpha}$, for all $\alpha$, has one zero eigenvalue while others are strictly positive. This zero eigenvalue is corresponding to an eigenvector $\1_{N_{\alpha}}$. Note that $\1_{N_{\alpha}}$ is in the null space of $\Wbf_{\alpha}$. Thus, since $\ebf^{x}_{\alpha} = \Wbf_{\alpha}\xbf_{\alpha}\in\1^{\perp}$ and using \eqref{notaion:sigma_2_I} we have 
\begin{align}
-(\ebf^{x})^T\Lbf\ebf^{x} \leq -\sigma_{2}^{I}\|\ebf^{x}\|^2. \label{Laplacian_ineq1}
\end{align}
Similarly, we obtain 
\begin{align}
-(\ebf^{y})^T\widetilde{\Lbf}^{E}\ebf^{y} \leq -\sigma_{2}(\widetilde{\Lbf}^{E})\|\ebf^{y}\|^2.\label{Laplacian_ineq2}
\end{align}
Here, $\sigma_{2}^{I}$ and $\sigma_{2}(\widetilde{\Lbf}^{E})$ represent the algebraic connectivity of $\Lbf^{I}$ and $\widetilde{\Lbf}^{E}$, respectively. Finally, we consider the following assumption about the cluster structure of the network.
\begin{assump}\label{assump:cluster}
$G$ has a cluster structure that satisfies
\begin{align}
\sigma_{2}^{I} &\geq \Big(\frac{12L}{\mu } + \|\Lbf^{E}\|\Big) \frac{N_{\max}\|\Lbf^{E}\|}{N_{\min}\sigma_{2}(\widetilde{\Lbf}^{E})}.\label{assump:cluster_condition}    
\end{align}
\end{assump}
\begin{remark}
This assumption, similar to the one in \cite{biyik2008area} (see Section $4.2$) and \cite{Dutta_cluster_consensus2021}, basically implies that the internal connections within any cluster is much denser than the external connections across clusters. To see this, consider for simplicity $L = \mu$ and $N_{\min} = N_{\max}$. Thus, \eqref{assump:cluster_condition} gives
\begin{align}\label{nw_condition}
\sigma_{2}^{I} &\gtrsim \frac{\|\widetilde{\Lbf}^{E}\|^2}{\sigma_{2}(\widetilde{\Lbf}^{E})},
\end{align}
which implies that the number of internal connections  represented by $\sigma_{2}^{I}$ is greater than the number of external connections represented by $\|\widetilde{\Lbf}^{E}\|$. In Section \ref{sec:simulations}, we will verify that this condition holds in our simulations. 
\end{remark}
We next consider the following useful lemmas about the Lyapunov functions in \eqref{notation:Lyapunov_func}. For convenience, we present the proofs of these lemmas in the Appendix. 

\begin{lem}\label{lem:V_ey_dot}
The Lyapunov function $V(\ebf_{y})$  satisfies 
\begin{align}
     \hspace{-0.25cm}\frac{dV(\ebf^{y})}{dt} &\leq  \frac{-\sigma_{2}(\widetilde{\Lbf}^{E})}{N_{\max}}V(\ebf^{y}) + \frac{\|\Lbf^{\Ebf}\|}{N_{min}}V(\ebf^{x})  + \frac{L\gamma(t)}{N_{\min}}\cdot\label{lem:V_ey_dot:ineq}
\end{align}
\end{lem}
\begin{lem}\label{lem:V_ex_dot}
The Lyapunov function $V(\ebf^{x})$ satisfies
\begin{align}
         \frac{dV(\ebf^{x})}{dt} &\leq -\sigma_{2}^{I}V(\ebf^{x}) + \|\Lbf^{E}\|V(\ebf^{y}) + L\gamma(t).\label{lem:V_ex_dot:ineq}
\end{align}
\end{lem}
\begin{lem}\label{lem:lypunov_V_delta}
For any $p\in[1,N]$, $V(\xbar(t))$ satisfies
\begin{align}
    \frac{d V(\xbar(t))}{dt} &\leq -\frac{\mu\gamma(t)}{2} \|\xbar-x^{\star}\|^2 + \gamma(t)[f(x^{\star}) - f(x_{p})]\notag\\
    &\quad + 3L\gamma(t)(V(\ebf^{y}) + V(\ebf^{x})),\label{lem:lypunov_V_delta:ineq}
\end{align}
\end{lem}
We now present the main result of this paper, which is the convergence rate of \eqref{alg:DCG} over cluster networks. 
\begin{thm}
Suppose that Assumptions \ref{assump:graph_connectivity}--\ref{assump:cluster} hold and let $x_{i}(t)$, for all $i\in\Vcal$, be generated by \eqref{alg:DCG}. Let $\epsilon$ and $D$ be defined, respectively, as
\begin{align}
\begin{aligned}
\epsilon &= \frac{6LN_{\max}}{\mu\sigma_{2}(\widetilde{\Lbf}^{E})} + \frac{N_{\max}\|\Lbf^{E}\|}{\sigma_{2}(\widetilde{\Lbf}^{E})},\\
D &=  V(\ebf^{x}(0)) + \epsilon V(\ebf^{y}(0)).
\end{aligned}\label{thm:epsilon}
\end{align}
Let each node $p\in V$ maintain
\begin{align}
z_{p}(T) =   \frac{1}{T}\int_{t = 0}^{T} x_{p}(t)dt.\label{thm:rate:z} 
\end{align}
Then we have for any $p\in V$ and $T\geq 0$
\begin{align}
\begin{aligned}
&\|z_{p}(T)-z^{\star}\|^2\leq \frac{D}{T} + \frac{2L(N_{\min}+\epsilon)}{\mu N_{\min}}\frac{\ln(T)}{T}\cdot
\end{aligned}\label{thm:rate:ineq}    
\end{align}
\end{thm}
\begin{remark}
Our result in \eqref{thm:rate:ineq} shows that DCG algorithm converges to the optimal solution at a common rate $\ln(T)/T$. In addition, this convergence scales with $\epsilon$ that appears in both constant $D$ and the second term. The factor $\epsilon$ is proportional to $1/\sigma_{2}(\widetilde{\Lbf}^{E})$, that is, the convergence of DCG when implementing over cluster networks only scales with a small number of nodes in the sparse graph $G_{E}$ that have external connections to other clusters. On the other hand, if we apply the existing analysis one would expect that DCG has convergence time proportional to $1/\sigma_{2}(\Lbf^{E})$ (or spectral gap in discrete time \cite{Nedic_review2018}), which scales with $N$.
\end{remark}
\begin{proof}
Consider the following coupled Lyapunov function
\begin{align*}
&\frac{d }{dt}(tV(\xbar)) + \frac{d }{dt}V(\ebf^{x}) + \epsilon \frac{d }{dt}V(\ebf^{y})\notag\\
&= V(\xbar(t))  + t\frac{dV(\xbar(t))}{dt} + \frac{d }{dt}V(\ebf^{x}) + \epsilon \frac{d }{dt}V(\ebf^{y}),
\end{align*}
which by using \eqref{lem:V_ey_dot:ineq}--\eqref{lem:lypunov_V_delta:ineq} we obtain  
\begin{align}
&\frac{d }{dt}(tV(\xbar)) + \frac{d }{dt}V(\ebf^{x}) + \epsilon \frac{d }{dt}V(\ebf^{y})\notag\\
&= \big( 1 -  \frac{\mu \gamma(t) t}{2}\big)V(\xbar(t))+ t\gamma(t)[f(x^{\star}) - f(x_{p})]\notag\\
&\quad + 3Lt\gamma(t)(V(\ebf^{y}) + V(\ebf^{x}))\notag\\ 
&\quad -\sigma_{2}^{I}V(\ebf^{x}) + \|\Lbf^{E}\|V(\ebf^{y}) + L\gamma(t)\notag\\
&\quad \frac{-\sigma_{2}(\widetilde{\Lbf}^{E})\epsilon}{N_{\max}}V(\ebf^{y}) + \frac{\|\Lbf^{\Ebf}\|\epsilon}{N_{min}}V(\ebf^{x})  + \frac{L\epsilon\gamma(t)}{N_{\min}}\notag\\
&= \big( 1 -  \frac{\mu \gamma(t) t}{2}\big)V(\xbar(t))+ t\gamma(t)[f(x^{\star}) - f(x_{p})]\notag\\
&\quad - \Big(\frac{\sigma_{2}(\widetilde{\Lbf}^{E})\epsilon}{N_{\max}} - \|\Lbf^{E}\| - 3Lt\gamma(t)\Big)V(\ebf^{y})\notag\\
&\quad - \Big(\sigma_{2}^{I} - 3Lt\gamma(t) -  \frac{\|\Lbf^{\Ebf}\|\epsilon}{N_{min}} \Big)V(\ebf^{x})\notag\\
&\quad + \Big(1+\frac{\epsilon}{N_{\min}}\Big)L\gamma(t).\label{thm:rate:Eq1}
\end{align}
Recall that $\gamma(t) = 2/(\mu t)$ and 
\begin{align*}
\epsilon &=  \frac{6LN_{\max}}{\mu\sigma_{2}(\widetilde{\Lbf}^{E})} + \frac{N_{\max}\|\Lbf^{E}\|}{\sigma_{2}(\widetilde{\Lbf}^{E})},
\end{align*}
which when substituting into Eq.\ \eqref{thm:rate:Eq1} yields 
\begin{align*}
&\frac{d }{dt}(tV(\xbar)) + \frac{d }{dt}V(\ebf^{x}) + \epsilon \frac{d }{dt}V(\ebf^{y})\notag\\
&\leq \frac{2}{\mu}[f(x^{\star}) - f(x_{p})] + \Big(1+\frac{\epsilon}{N_{\min}}\Big)\frac{2L}{\mu t}\notag\\
&\quad  - \Big(\sigma_{2}^{I} - \frac{6L}{\mu} -  \frac{\|\Lbf^{\Ebf}\|\epsilon}{N_{min}} \Big)V(\ebf^{x}).
\end{align*}
Using Assumption \ref{assump:cluster} into the preceding relation yields
\begin{align*}
&\frac{d }{dt}(tV(\xbar)) + \frac{d }{dt}V(\ebf^{x}) + \epsilon \frac{d }{dt}V(\ebf^{y})\notag\\
&\leq   \frac{2}{\mu}[f(x^{\star}) - f(x_{p})] + \Big(1+\frac{\epsilon}{N_{\min}}\Big)\frac{2L}{\mu t},
\end{align*}
which when taking integral on both sides from $0$ to $T$ and using $D$ in \eqref{thm:epsilon} gives
\begin{align*}
&TV(\xbar(T)) + V(\ebf^{x}(T)) + \epsilon V(\ebf^{y}(T))  \notag\\
&\leq D + \frac{(N_{\min}+\epsilon)}{N_{\min}}\frac{2L\ln(T)}{\mu} + \frac{2}{\mu}\int_{t = 0}^{T}[f(x^{\star}) - f(x_{p}(t))]dt.
\end{align*}
Dividing both sides by $T$ and rearranging we have
\begin{align*}
&\frac{2}{\mu T}\int_{t = 0}^{T}[f(x_{p}(t)) - f(x^{\star})]dt\notag\\
&\quad \leq \frac{D}{ T} + \frac{2L(N_{\min}+\epsilon)}{\mu N_{\min}}\frac{\ln(T)}{T}, 
\end{align*}
which by using the convexity of $f$ and \eqref{thm:rate:z} gives
\begin{align*}
\frac{2}{\mu}\Big[f\big(z_{p}(T)\big) - f(x^{\star})\Big]\leq \frac{D}{ T} + \frac{2L(N_{\min}+\epsilon)}{\mu N_{\min}}\frac{\ln(T)}{T}\cdot    
\end{align*}
Since $f$ is strongly convex and $\nabla f(x^{\star}) = 0$ we have
\begin{align*}
\|z_{p}(T)-z^{\star}\|^2  \leq \frac{2}{\mu}\Big[f\big(z_{p}(T)\big) - f(x^{\star})\Big],  
\end{align*}
which concludes our proof. 
\end{proof}
\begin{remark}
One can easily extend the results in this section to the multi-dimensional setting, i.e., $d>1$. In this case, each $\xbf_{i}$ is a vector in $\Rset^{d}$ and $X$ and $\nabla F(\Xbf)$ are matrices defined as
\begin{align*}
\Xbf = \left[\begin{array}{c}
     \xbf_{1}^T \\
     \vdots\\
     \xbf_{N}^T 
\end{array}
\right],\quad \nabla F(\Xbf) = \left[\begin{array}{c}
     \nabla f_{1}(\xbf_{1})^T \\
     \vdots\\
     \nabla f_{N}(\xbf_{N})^T 
\end{array}
\right].    
\end{align*}
This gives 
\begin{align*}
\frac{d\Xbf(t)}{dt} = -\Lbf\Xbf(t) - \gamma(t)\nabla F(\Xbf(t)).      
\end{align*}
Similarly, other vector notation (e.g., $\Ybf,\ebf^{x},\ebf^{y}$) become matrix notation. However, our analysis will remain the same. 
\end{remark}


\section{Simulations}\label{sec:simulations}
In this section we support the theoretical findings presented in the previous section by a number of numerical simulations. We apply a discrete-time variant of DCG method in \eqref{Eq:Consesus_dynamics_Laplacian} over different cluster networks to study the following 
\begin{itemize}
    \item Dependency of convergence of dynamics of DCG method on connectivity between the clusters.\label{itm:goal1}
    \item Impact of network scalability on the convergence of DCG over the network. \label{itm:goal2}
\end{itemize}
For this we assume that each node has its own estimate $x_{i}\in \mathbb{R}^{d}$ and also maintains the parameters $(\mathbf{A}_{i},b_{i})$ with $\mathbf{A}_{i} \in \mathbb{R}^{l\times d}$ and $b_{i} \in \mathbb{R}^{l}$. Each node has a quadratic function $f_{i}$
\begin{figure}[t!]
    \centering
    \includegraphics[width=0.7\columnwidth]{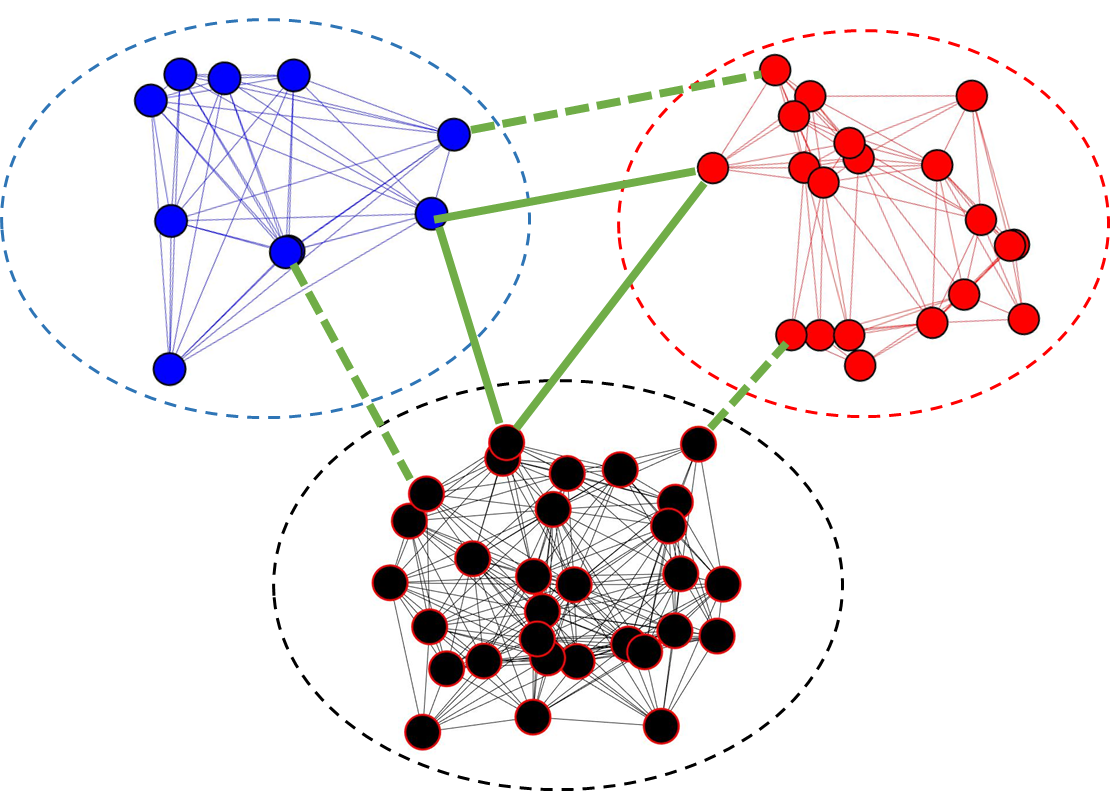}
\caption{A 60-nodes network divided into 3 clusters.}
    \label{fig:Simulation_1_2_clusters}
\end{figure}
\begin{figure}[t!]
    \centering
    \includegraphics[width=\columnwidth]{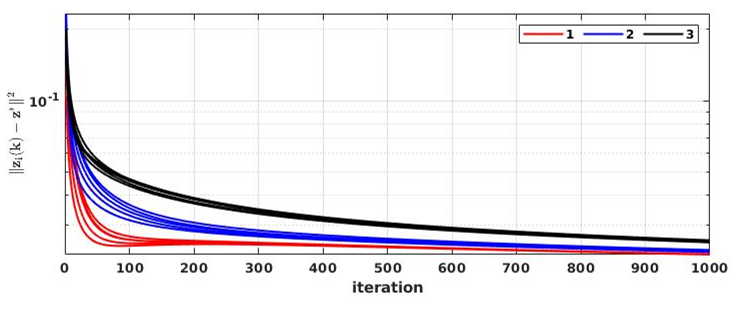}
    \caption{\textit{Simulation 1} for the network in Fig.\ \ref{fig:Simulation_1_2_clusters} using 3 external communication links.}
    \label{fig:Simulation_3_plot}
\end{figure}
\begin{align*}
   f_{i}(x) = \|\mathbf{A}_{i}x - b_{i}\|^{2}.
\end{align*}
The goal of the network is to optimize the mean of all the functions held by the nodes. Thus the distributed optimization problem for the given network can be formulated as 
\begin{align*}
    \min_{x}f(x) = \min_{x}\sum_{i=1}^{N}f_{i}(x)= \min_{x}\sum_{i=1}^{N}\|\mathbf{A}_{i}x - b_{i}\|^{2}.
\end{align*}
For our simulation, we consider step size $\gamma(k) = 1/k$ and $\mathbf{W}$ as a metropolis adjacency matrix given by
\begin{align*}
    \mathbf{W} = [w_{ij}]  = \left\{\begin{array}{ll}
\frac{1}{(\max(\mathcal{N}_{i},\mathcal{N}_{j}))} & \ \text{if} \ (i,j) \in E,\\
0 &  \text{if} (i,j) \notin E \ \& \ i \neq j,\\
1 - \sum_{j \in \mathcal{N}_{i}}w_{ij} & \text{otherwise.} 
\end{array}\right. 
\end{align*}
\subsection{Small network}   
We consider a cluster network with 60 nodes having 3 distinct clusters as shown in Fig. \ref{fig:Simulation_1_2_clusters}. The number of nodes of the clusters are $N_{1}=10$, $N_{2}=20$ and $N_{3}=30$. Here we present two simulations whose results are obtained from two different networks. First,  we consider a network in which the clusters communicate with each other through $3$ links using only solid green links in Fig. \ref{fig:Simulation_1_2_clusters}. Second, we consider $6$ communications links between the clusters using both solid and dashed green links.\\
\begin{figure}[t]
    \centering
    \includegraphics[width=\columnwidth]{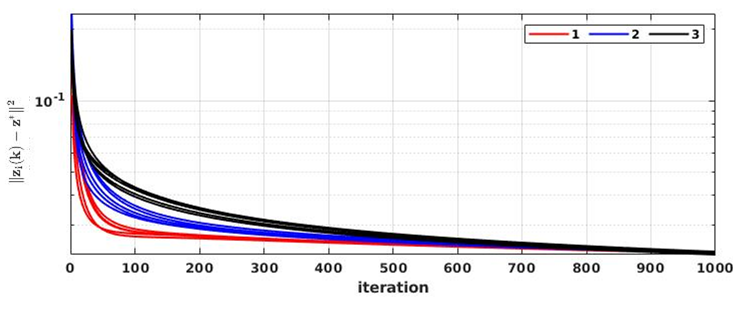}
    \caption{\textit{Simulation 2} for the network in Fig.\ \ref{fig:Simulation_1_2_clusters} using 6 external communication links.}
    \label{fig:Simulation_6_plot}
\end{figure}
We first verify Assumption \ref{assump:cluster} is satisfied in our setting. We have $N_{min} = 10$ and $N_{max} =30$. The eigenvalues of the different Laplacian matrices are $\sigma_{2}(\mathbf{L}^{I}_{1}) = 6.79,\, \sigma_{2}(\mathbf{L}^{I}_{2}) = 16.48\ \text{and} \ \sigma_{2}(\mathbf{L}^{I}_{3}) = 6.52$. We also have $\sigma_{2}(\widetilde{\mathbf{L}}) = 0.1$ and $\sigma_{2}(\widetilde{\mathbf{L}}^{E}) = 3$. It is straight forward to verify that the condition \eqref{nw_condition} in Assumption \eqref{assump:cluster} holds.

Simulation results for both the networks are shown in Fig. \ref{fig:Simulation_3_plot} and \ref{fig:Simulation_6_plot}, respectively. We present the optimal gap of the variables, $\|z_{p}(k)-z^{\star}\|$, for any $p^{th}\in\Vcal$. We observe that the nodes first perform a local synchronization, then they slowly converge to the optimal solution, which agrees with our theoretical results. This reflects the two time scale nature of the cluster network. In addition, we observe that the network with $3$ external communication links converge slower than the one with $6$ external communication links. 
\subsection{Large network}
We present our next simulation result to show the performance of DCG algorithm over a large cluster network. Here we consider $5$ clusters, each having $60$ nodes. All the clusters are generated in the similar manner as done in simulations $1$ and $2$. The external communications between the clusters are represented as blue links as shown in Fig. \ref{fig:Simulation_5_clusters}. We observe similar behaviors of DCG over large networks as compared to small networks above.

\begin{figure}[t!]
    \centering
    \includegraphics[width=0.6\columnwidth]{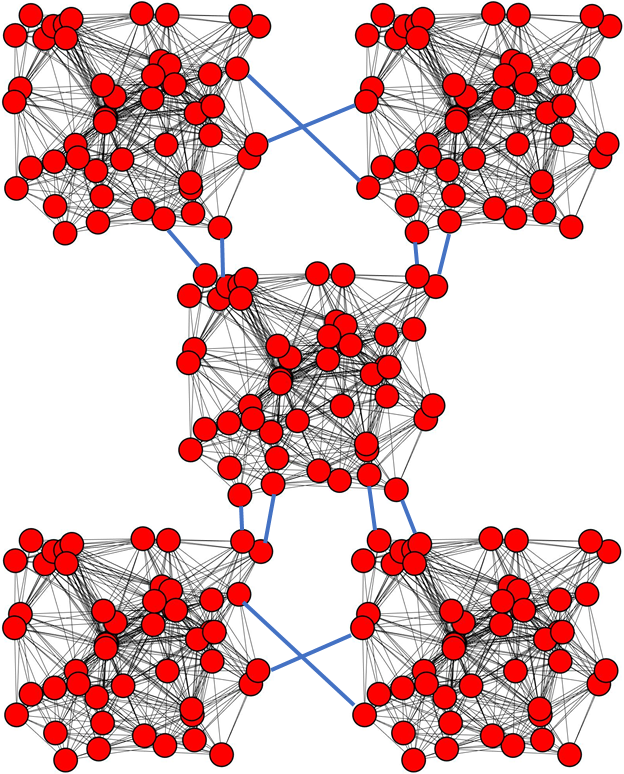}
    \caption{A 300-nodes network divided equally into 5 clusters. }
    \label{fig:Simulation_5_clusters}
\end{figure}
\begin{figure}[h]
    \centering
    \includegraphics[width=\columnwidth]{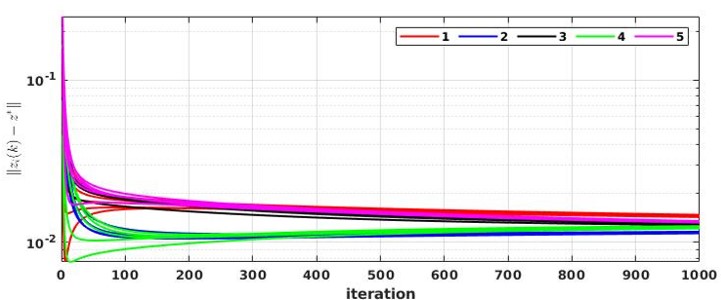}
    \caption{\textit{Simulation 3} for the network in Fig. \ref{fig:Simulation_5_clusters}.}
    \label{fig:Simulation_5_plot}
\end{figure}



\bibliographystyle{IEEEtran}
\bibliography{IEEEfull,ncs}

\appendix
\subsection{Proof of Lemma \ref{lem:V_ey_dot}}
\begin{proof}
Recall that $V(\ebf^{y}(t))$ be 
\begin{align*}
V(\ebf^{y}(t)) = \|\ebf^{y}(t)\|,    
\end{align*}
which gives
\begin{align}\label{lem:V_ey_dot:eq0}
&\frac{d V(\ebf^{y})}{dt} = \frac{(\ebf^{y})^{T}\dot{\mathbf{e}}_{y}}{\|\ebf^{y}(t)\|}\cdot
\end{align}
By \eqref{eq:dynamics_inter_cluster} in Lemma \ref{dynamics_inter_cluster_lem} we consider 
\begin{align}
(\ebf^{y})^{T}\dot{\mathbf{e}}_{y} &= -(\ebf^{y})^{T}\mathbf{P}^{-1}\widetilde{\Lbf}^{E}\ebf^{y} - (\ebf^{y})^{T}\mathbf{P}^{-1}\mathbf{U}^{T}\mathbf{L}^{E}\mathbf{e}^{x}, \notag \\
&\quad -\gamma(t) (\ebf^{y})^{T}\Big[\mathbf{P}^{-1}\mathbf{U}^{T}-\frac{1}{N}\1_{r}\1_{N}^T\Big]\nabla F(\mathbf{x}).\label{lem:V_ey_dot:eq1}
\end{align}
First, using the definitions of $\Pbf$ and $\Ubf$ in \eqref{notation:U_P} we have
\begin{align*}
     &\frac{1}{N_{\max}}\mathbf{I} \leq \mathbf{P}^{-1} \leq \frac{1}{N_{\min}}\mathbf{I}\\
     &\Big\|\mathbf{U}^{T}-\frac{1}{N}\Pbf\1_{r}\1_{N}^T\Big\| \leq 1.
\end{align*}
Second, we consider the second term in \eqref{lem:V_ey_dot:eq1}  
\begin{align}
& -(\ebf^{y})^T\mathbf{P}^{-1}\mathbf{U}^{T}\mathbf{L}^{E}\mathbf{e}^{x} \leq \|\ebf^{y}\|\|\mathbf{P}^{-1}\mathbf{U}^{T}\mathbf{L}^{E}\mathbf{e}^{x} \|\notag\\ 
&\quad \leq \frac{\|\Lbf^{\Ebf}\|}{N_{min}}\|\ebf^{y}\|\|\mathbf{e}^{x} \|.\label{lem:V_ey_dot:eq1a}
\end{align}
Third, using \eqref{Laplacian_ineq2} we consider the first term in \eqref{lem:V_ey_dot:eq1}
\begin{align}
-(\ebf^{y})^{T}\mathbf{P}^{-1}\widetilde{\Lbf}^{E}\ebf^{y} \leq \frac{-\sigma_{2}(\widetilde{\Lbf}^{E})}{N_{\max}}\|\ebf^{y}\|^2\cdot\label{lem:V_ey_dot:eq1b} 
\end{align}
Fourth, using \eqref{assump:obj:func_Lipschitz} we consider the last term in \eqref{lem:V_ey_dot:eq1} 
\begin{align*}
&-\gamma(t) (\ebf^{y})^{T}\Big[\mathbf{P}^{-1}\mathbf{U}^{T}-\frac{1}{N}\1_{r}\1_{N}^T\Big]\nabla F(\mathbf{x})\notag\\
&\leq \gamma(t) \|\ebf^{y}\|\mathbf{P}^{-1}\Big\|\mathbf{U}^{T}-\frac{1}{N}\mathbf{P}\1_{r}\1_{N}^T\Big\|\|\nabla F(\mathbf{x})\|\notag\\
&\leq \frac{L\gamma(t)}{N_{\min}} \|\ebf^{y}\|.
\end{align*}
Substituting this relation,  \eqref{lem:V_ey_dot:eq1a}, and \eqref{lem:V_ey_dot:eq1b} into \eqref{lem:V_ey_dot:eq1} gives
\begin{align*}
(\ebf^{y})^{T}\dot{\mathbf{e}}_{y} &\leq \frac{-\sigma_{2}(\widetilde{\Lbf}^{E})}{N_{\max}}\|\ebf^{y}\|^2 + \frac{\|\Lbf^{\Ebf}\|}{N_{min}}\|\ebf^{y}\|\|\mathbf{e}^{x} \|\notag\\ 
&\quad + \frac{L\gamma(t)}{N_{\min}} \|\ebf^{y}\|,
\end{align*}
where when substituting into \eqref{lem:V_ey_dot:eq0} gives \eqref{lem:V_ey_dot:ineq}. 
\end{proof}

\subsection{Proof of Lemma \eqref{lem:V_ex_dot}}
\begin{proof}
Recall that $V(\ebf^{y}(t))$ be 
\begin{align*}
V(\ebf^{x}(t)) = \|\ebf^{x}(t)\|,    
\end{align*}
which gives
\begin{align}
\frac{d V(\ebf^{x})}{dt} = \frac{(\ebf^{x})^{T}\dot{\mathbf{e}}_{x}}{\|\ebf^{x}(t)\|}\cdot\label{lem:V_ex_dot:eq0}
\end{align}
Using Lemma \ref{lem:fast_slow_dynamics} we consider
\begin{align}
 (\ebf^{x})^{T}\dot{\mathbf{e}}_x
&=-(\ebf^{x})^T\mathbf{W}\mathbf{L}^{I}\mathbf{e}^{x} - (\ebf^{x})^T\mathbf{W}\mathbf{L}^{E}\mathbf{e}^{x}\notag \\
&\quad - (\ebf^{x})^T\mathbf{W}\mathbf{L}^{E}\mathbf{U}\ebf^{y} -\gamma(t)\mathbf{e}^{x}(t)\mathbf{W}\nabla F(\mathbf{x}).\label{lem:V_ex_dot:eq1}
\end{align}
First, by \eqref{Laplacian_ineq1} we have 
\begin{align}
-(\ebf^{x})^T\mathbf{W}\mathbf{L}^{I}\mathbf{e}^{x} \leq -\sigma_{2}^{I}\|\ebf^{x}\|^2.   \label{lem:V_ex_dot:eq1a}  
\end{align}
Second, we consider the third term in \eqref{lem:V_ex_dot:eq1}
\begin{align}
&- (\ebf^{x})^T\mathbf{W}\mathbf{L}^{E}\mathbf{U}\ebf^{y} \leq \|\ebf^{x}\|\|\mathbf{W}\mathbf{L}^{E}\mathbf{U}\|\|\ebf^{y}\|\notag\\
&\leq \|\Lbf^{E}\|\|\ebf^{x}\|\|\ebf^{y}\|. \label{lem:V_ex_dot:eq1b}
\end{align}
Third, using \eqref{assump:obj:func_Lipschitz} we consider
\begin{align*}
 -\gamma(t)\mathbf{e}^{x}(t)\mathbf{W}\nabla F(\mathbf{x}) \leq  L\gamma(t)\|\ebf^{x}\|.
\end{align*}
Substituting the preceding equation, \eqref{lem:V_ex_dot:eq1a}, and \eqref{lem:V_ex_dot:eq1b} into  \eqref{lem:V_ex_dot:eq1}, and using $\Wbf\Lbf^{E} > 0$ we obtain
\begin{align*}
 (\ebf^{x})^{T}\dot{\mathbf{e}}_x 
&\leq -\sigma_{2}^{I}\|\ebf^{x}\|^2 + \|\Lbf^{E}\|\|\ebf^{x}\|\|\ebf^{y}\| + L\gamma(t)\|\ebf^{x}\|,
\end{align*}
which when substituting into \eqref{lem:V_ex_dot:eq0} gives \eqref{lem:V_ex_dot:ineq}.  
\end{proof}

\subsection{Proof of Lemma \ref{lem:lypunov_V_delta}}
\begin{proof}
Using \eqref{Eq:Consesus_dynamics_Laplacian} we consider 
\begin{align}\label{Eq:V_del_2:eq1}
    \frac{d V(\xbar(t))}{dt} &= \dot{\bar{x}}^{T}(\bar{x}-x^{\star}) \notag\\
    &=- \frac{\gamma(t)}{N}\overset{N}{\sum_{i=1}} \nabla f_{i}(x_{i} )^{T}(\bar{x} -x^{\star}),\notag\\
    &= \underbrace{- \frac{\gamma(t)}{N}\overset{N}{\sum_{i=1}} \nabla f_{i}(x_{i})^{T}(\bar{x} -x_{i})}_{A}\notag\\
    &\quad \underbrace{- \frac{\gamma(t)}{N}\sum_{i=1}^{N}\nabla f_{i}(x_{i})^{T}(x_{i} - x^{\star})}_{B}.
\end{align}
First, we analyze $A$ using \eqref{assump:obj:func_Lipschitz}
\begin{align*}
A &=    - \frac{\gamma(t)}{N}\overset{N}{\sum_{i=1}} \nabla f_{i}(x_{i})^{T}(\bar{x} -x_{i})\notag\\
&\leq \frac{\gamma(t)}{N}\sum_{i=1}^{N}\|\nabla f_{i}(x_{i})\|\|\bar{x} -x_{i}\|\leq \frac{L\gamma(t)}{N}\sum_{i=1}^{N}\|x_{i} - \xbar \|. \end{align*}
Next, using the strong convexity of $f_{i}$ we consider $B$ in \eqref{Eq:V_del_2:eq1} 
\begin{align*}
    B &\leq -\frac{\gamma(t)}{N}\sum_{i=1}^{N}\Big(f_{i}(x_{i}) - f_{i}(x^{\star}) + \frac{\mu}{2}\|x_{i} - x^{\star}\|^{2}\Big)\notag\\
    &= -\frac{\gamma(t)}{N}\sum_{i=1}^{N}\Big(f_{i}(x_{i}) - f_{i}(x_{p}) + f_{i}(x_{p}) - f_{i}(x^{\star})\Big) \notag\\
    &\quad -\frac{\mu\gamma(t)}{2N}\sum_{i=1}^{N} \|x_{i}-x^{\star}\|^2\notag\\
    &= \gamma(t)[f(x^{\star}) - f(x_{p})] -\frac{\mu\gamma(t)}{2N}\sum_{i=1}^{N} \|x_{i}-x^{\star}\|^2\notag\\ 
    &\quad -\frac{\gamma(t)}{N}\sum_{i=1}^{N}\Big(f_{i}(x_{i}) - f_{i}(x_{p})\Big) \notag\\
    &\leq \gamma(t)[f(x^{\star}) - f(x_{p})] -\frac{\mu\gamma(t)}{2} \|\xbar-x^{\star}\|^2\notag\\
    &\quad + \frac{L\gamma(t)}{N}\sum_{i=1}^{N}\|x_{i} - x_{p}\|,
\end{align*}
where $p$ is some fixed index in $[1,N]$, and the second last inequality is due to the Cauchy-Schwarz inequality and \eqref{assump:obj:func_Lipschitz}. Substituting the preceding two relations into \eqref{Eq:V_del_2:eq1} we obtain
\begin{align}
\frac{d V(\xbar)}{dt} &= \gamma(t)[f(x^{\star}) - f(x_{p})] -\frac{\mu\gamma(t)}{2} \|\xbar-x^{\star}\|^2\notag\\ &\quad + \frac{L\gamma(t)}{N}\sum_{i=1}^{N}\|x_{i} - x_{p}\| + \frac{L\gamma(t)}{N}\sum_{i=1}^{N}\|x_{i} - \xbar \|. \label{lem:lypunov_V_delta:eq1}
\end{align}
We note that $p\in C_{\alpha}$ for some $\alpha$. Then we have
\begin{align*}
\|\xbar-x_{p}\| \leq \|\xbar-\xbar_{\alpha}\| + \|\xbar^{\alpha}-x_{p}\| \leq \|\ebf^{x}\| + \|\ebf^{y}\|.     
\end{align*}
Thus, we consider
\begin{align*}
& \frac{L\gamma(t)}{N}\sum_{i=1}^{N}\|x_{i} - x_{p}\| + \frac{L\gamma(t)}{N}\sum_{i=1}^{N}\|x_{i} - \xbar \|\notag\\
&\leq  \frac{L\gamma(t)}{N}\sum_{i=1}^{N}\Big(\|x_{i} - \xbar\| + \|\xbar-x_{p}\|\Big)\notag\\ 
&\quad + \frac{2L\gamma(t)}{N}\sum_{i=1}^{N}\|x_{i} - \xbar \|\notag\\
&\leq L\gamma(t)(\|\ebf^{x}\| + \|\ebf^{y}\|) +  \frac{2L\gamma(t)}{N}\sum_{i=1}^{N}\|x_{i} - \xbar\|\notag\\
&= L\gamma(t)(\|\ebf^{x}\| + \|\ebf^{y}\|) +  \sum_{\alpha=1}^{r}\sum_{i=1}^{N_{\alpha}}\|x_{i} - \xbar\|\notag\\
&\leq L\gamma(t)(\|\ebf^{x}\| + \|\ebf^{y}\|)\notag\\
&\quad + \frac{2L\gamma(t)}{N} \sum_{\alpha=1}^{r}\sum_{i=1}^{N_{\alpha}}\Big(\|x_{i} - \xbar_{\alpha}\| + \|\xbar_{\alpha} -\xbar\|\notag\\
&\leq 3L\gamma(t)(\|\ebf^{x}\| + \|\ebf^{y}\|),
\end{align*}
which when substituting into \eqref{lem:lypunov_V_delta:eq1} yields \eqref{lem:lypunov_V_delta:ineq}. 
\end{proof}

\end{document}